\documentclass[12pt,twoside]{amsart}
\usepackage{verbatim}
\usepackage{amsmath}
\usepackage{amsthm}
\usepackage{amssymb}
\usepackage{enumerate}
\usepackage{psfrag}
\usepackage{graphicx}
\newif\ifdeveloping

\ifdeveloping
\usepackage[notref,notcite]{showkeys}
\fi

\newif\ifcommented

\newcommand{\comm}[1]{}

\ifcommented
\renewcommand{\comm}[1]{
\fbox{\fbox{\begin{minipage}{300pt}#1\end{minipage}}
}}

\fi

\newcommand{\prtime}{{\count0=\time\divide\count0 by 60
    \count1=-\count0\multiply\count1 by 60 \advance\count1 by \time
    \the\count0:\the\count1} }

\def\myheads#1;#2;{
    \pagestyle{myheadings} \markboth{{\sc\hfill
    #1\hfill\protect\makebox[0cm][r]{\rm\today; \prtime}}}
    {{\sc\protect\makebox[0cm][l]{\rm\today;\ \prtime}\hfill
    #2\hfill}} \thispagestyle{myheadings} }

\def\myheadsa#1;#2;{
    \pagestyle{myheadings} \markboth{{\sc\hfill #1\hfill}}
    {{\sc\hfill #2\hfill}} \thispagestyle{myheadings} }

\let\QED\qed
\newcommand{\prlabel}[1]{\renewcommand{\qed}{\QED${}_{\ref{#1}}$}}

\newtheorem{theorem}{Theorem}[section]
\newtheorem{corollary}[theorem]{Corollary}

\newtheorem{lemma}[theorem]{Lemma}

\newtheorem{problem}[theorem]{Problem}
\newtheorem{claim}[theorem]{Claim}

\newtheorem{conjecture}[theorem]{Conjecture}

\newtheorem{case}{Case}

\newtheorem{rtheorem}{Theorem}  
\newtheorem{rconjecture}{Conjecture}

\def\rem{\medskip \noindent {\bf Remark.}}

\theoremstyle{definition}
\newtheorem{definition}[theorem]{Definition}

\def\<{\left\langle}
\def\>{\right\rangle}

\def\br#1;#2;{\bigl[ {#1} \bigr]^ {#2} }
\def\bc#1;#2;{\bigl( {#1} \bigr)^ {#2} }

\newcommand{\subs}{\subset}
\newcommand{\setm}{\setminus}
\newcommand{\empt}{\emptyset}

\newcommand{\inc}[3]{\operatorname{{In}^{#1}_{#2}}({#3})}
\newcommand{\outc}[3]{\operatorname{{Out}}^{#1}_{#2}({#3})}

\newcommand{\incl}[1]{\operatorname{\mathfrak {In}}_{#1}}
\newcommand{\outcl}[1]{\operatorname{\mathfrak {Out}}_{#1}}

\newcommand{\inout}[2]{\mathfrak{In}_{#1}\text{-}%
\mathfrak{Out}_{#2}}

\newcommand{\T}{\mathbb T}
\newcommand{\testinf}{\T_{\infty}}
\newcommand{\testinfk}{\T_{{\kappa},{\infty}}}
\newcommand{\testt}{\T_{3}}

\newcommand{\ecal}{\mathcal E}

\newcommand{\grt}{terminated digraph}
\newcommand{\grts}{terminated digraphs}
\newcommand{\ext}[1]{\overline{#1}}
\newcommand{\gra}[2]{({#1},{#2})}
\newcommand{\concat}{{}^\frown}
\newcommand{\compl}[1]{\widetilde{#1}}

\title{Quasi-kernels and quasi-sinks in infinite graphs}
\thanks
    {The preparation of this paper was partly supported by the
    Hungarian NSF, under contract Nos. T 37758, T37846, T61600, NK62321, AT048826}

\author[P. L. Erd{\H o}s]{P{\'e}ter L. Erd{\H o}s}
\author[L. Soukup]{Lajos Soukup}
\address{ R{\'e}nyi Institute of Mathematics,
    Hungarian Academy of Sciences,
    Budapest, P.O.~Box 127, H-1364 Hungary}

\email{elp@renyi.hu}
\email{soukup@renyi.hu}


\begin{document}
 \ifdeveloping\myheads{Quasi-kernels};{Quasi-kernels};\else
   \myheadsa{Erd\H os-Soukup};{Quasi-kernels}; \fi

\begin{abstract}
\noindent 
Given a directed graph $G=(V,E)$ an independent  set  $A\subs V$ 
is called  {\em quasi-kernel $($quasi-sink$)$ } iff for each point $v$ there is a path of length at
most $2$ from some point of $A$ to $v$   (from $v$ to some point of $A$).
Every finite directed graph has a
quasi-kernel. The plain generalization for infinite graphs fails, 
even for tournaments.
We investigate the following conjecture here:   for
any digraph $G=(V,E)$ there is a a partition 
$(V_0,V_1)$ of the vertex set such that the induced subgraph  
$G[V_0]$ has a quasi-kernel and  the induced subgraph  $G[V_1]$ has a quasi-sink.
\end{abstract}

\maketitle



\section{Introduction}

\noindent
Given a directed graph $G=(V,E)$ an independent  set  $A\subs V$ 
is called  {\em quasi-kernel (quasi-sink) } iff for each point $v$ there is a path of length at
most $2$ from some point of $A$ to $v$   (from $v$ to some point of $A$).
(The notions have a fairly extensive
literature: as a starting point see for example the following
papers: \cite{GL}, \cite{gutin}, \cite{JM}.)

The starting point of our investigation was the following theorem:
\begin{theorem}[Chv\'atal- Lov\'asz, \cite{CH}]\label{tm:cl}
Every finite digraph (i.e. directed graph) contains 
a quasi-kernel. 
\end{theorem}
\noindent Our aim is to find similar theorems for infinite digraphs. The plain
generalization of Theorem~\ref{tm:cl} fails even for infinite
tournaments: the tournament $\gra{\mathbb Z}{<}$ is a counterexample, 
where $\mathbb Z$
denotes the set of the integers,  and $(x,y)$ is an edge iff $x<y$.
However, not just for $\gra{\mathbb Z}{<}$ but for each tournament 
$G$ we have a partition $(V_0,V_1)$ of the vertex set of the tournament
such that the induced subgraph $G[V_0]$ has a quasi-kernel and 
the induced subgraph $G[V_1]$ has a quasi sink, see
Theorem \ref{tm:inft}. Moreover, all the infinite digraph we could construct 
have this property.  These observations led to formulate 
the following conjecture.

\begin{conjecture}\label{conjmain}
Given any digraph $G=(V,E)$ one can find a partition 
$(V_0,V_1)$ of the vertex set such that the induced subgraph 
$G[V_0]$ has a quasi-kernel and $G[V_1]$ has a quasi-sink.
\end{conjecture}

The paper is organized as follows.
In Section \ref{s:step} we prove some easy results showing
that digraphs that  resemble
to a finite graph 
have quasi-kernels:
digraphs with finite in-degrees (Corollary \ref{cor:cl}) and 
digraphs with finite chromatic number (Corollary \ref{cor:mono})
have quasi-kernels.

We prove the conjecture above   for  digraphs 
that resemble to a tournament (Theorems \ref{tm:kn}  and
\ref{tm:clocally}), or that are built up from simple blocks (Corollary
\ref{cor:part}). 

In Section \ref{s:tour} we study the structures of infinite tournaments,
especially of tournaments without quasi-kernels.   
For  $n\in \mathbb N$
denote $\outcl n$ the family of  digraphs $G=(V,E)$
having an independent  set  $A\subs V$ 
such that  for each point $v$ there is a path of length at
most $n$ from some point of $A$ to $v$.  
In Section \ref{s:tour} for each $n\ge 3$ we could  
characterize infinite tournaments from $\outcl n$, see
Theorem \ref{tm:nott}. This characterization will imply immediately
that  the classes 
$\outcl 3, \outcl 4, \dots$ contain the same tournaments!
On the other hand,  
we show that $\outcl 2$  and $\outcl 3$ contain different
tournaments  (see Theorem \ref{tm:o3no2}), 
but the proof demands the development 
a recursive method to construct  infinite digraphs from certain
finites ones   (see Section \ref{s:rec}). One could hope that this method may help to disprove our conjecture, but this is not the case, because, in Theorem 
\ref{tm:ginf} we will show that 
 all the
digraphs obtained by this method also satisfy  the Conjecture 
\ref{conjmain}
above.

Finally, in Section \ref{s:haj} we give a weak version  
of Theorem \ref{tm:cl} for infinite digraphs. This result is a joint work
with Andr\'as Hajnal, and it is included with his kind permission.

\medskip

We will use standard combinatorial and set-theoretical notations. 
If
$V$ is a set then let $V^*$ be the family of finite sequences of
elements of $V$. If $a,b \in V^*$ then denote $a\concat b$ the {\em
concatenation} of the two sequences. If $A,B\subs V^*$ let $
A\concat B =\{a\concat b:a\in A, b\in B\} $. If $x\in V^*$
write $A\concat x$ for $A\concat \{x\}$.  The family of two
elements subsets of $V$ will be denoted by $\br V;2;$.

If $G=(V,E)$ is a digraph and $W\subs V$ denote 
the induced subgraph of $G$ on $W$ by $G[W]$, i.e. 
$G[W]=(W,E\cap W\times W)$.

To simplify the formulation of our results we introduce the
following terminology. Assume that  $G=(V,E)$ is a digraph and
$A\subs V$. For $n\in \mathbb N$ let
\begin{eqnarray*}
\inc GnA=\{v\in V:\text{there is a path of length at most $n$ } \\
\text{which leads from $v$ to some points of $A$}\}
\end{eqnarray*}
and
\begin{eqnarray*}
\outc GnA=\{v\in V:\text{there is a path of length at most $n$ } \\
\text{which leads from some points of $A$ to $v$}\}.
\end{eqnarray*}
Put
\begin{displaymath}
\outc G\infty A=\bigcup\{\outc GnA:n\in {\mathbb N}\}
\end{displaymath}
and
\begin{displaymath}
\inc G\infty A=\bigcup\{\inc GnA:n\in {\mathbb N}\}.
\end{displaymath}
If $A=\{a\}$ we write $\inc Gn{a}$ for $\inc Gn{\{a\}}$, and $\outc
Gn{a}$ for $\outc Gn{\{a\}}$. We will omit the superscript $G$
provided the digraph is clear from the context.

Using the notation above we can rephrase the definition of the
the classes $\outcl 2, \outcl 3, \dots$ 
and we can define the classes $\outcl \infty$, 
$\incl 2$, $ \incl3, \dots$ and  $\incl \infty$ of digraphs as follows.  
For $n\in \mathbb N\cup \{\infty\}$ let a digraph $G=(V,E)$ be in $ \incl n$ 
iff 
there is an independent set $A\subs V$ such that $V=\inc GnA$, and
let $G\in \outcl n$  
iff there is an independent set $B\subs V$
such that $V=\outc GnB$. We will say that ``{\em $A$  witnesses
$G\in \incl n $}'' and ``{\em $B$  witnesses $G\in \outcl n $}''.

If $n,k\in \mathbb N\cup\{\infty\}$ define the class $\inout nk$ of
digraphs as follows: let $G\in \inout nk$ if and only if there is a
partition $(V_1, V_2)$ of the vertex set $V$ such that $G[V_1]\in
\incl n$ and $G[V_2]\in \outcl k$. We will say that ``{\em
$(V_1,V_2)$ witnesses $G\in \inout nk$}''.

\noindent Using this new terminology we can formulate  Theorem
of Chv\'atal and Lov\'asz and our Conjecture 
 as follows:

\begin{rtheorem}
Every finite digraph is in $\outcl 2$,  
\end{rtheorem}

\begin{rconjecture}
 Every digraph is in $\inout 22$. 
\end{rconjecture}

\section{When $G$ resembles to a finite graph}
\label{s:step}

\noindent By a standard application of 
 G\"odel's Compactness Theorem one can get the
following  straightforward corollary of Theorem \ref{tm:cl}   for
infinite graphs:
\begin{corollary}\label{cor:cl}
If
 every vertex has finite in-degree in a digraph  $G$ 
 then $G$ has a quasi-kernel.
\end{corollary}

Next we prove two stepping-up theorems. 
The first will imply immediately that every finitely chromatic
digraph has quasi-kernel.  
The second one will be applied mainly later, in the next section.

\begin{definition}
A directed graph $G$ is {\em  hereditary in $\outcl n$ $($hereditary
in $\inout mn$ $)$} if and only if the  induced subgraphs of $G$ are
all in $\outcl n$ (in $\inout mn$).
\end{definition}
\begin{theorem}\label{tm:general2}
Let $n\ge1$ and assume  that $G=(V,E)$ is  a directed graph
having a  partition
$(V_0,V_1,...,V_k)$ of $V$ such that
\begin{enumerate}[{\rm (i)}]
\item $G[V_0]$ is hereditary in  $\outcl {n+1}$,
\item  $G[V_i]$ is hereditary in $\outcl n $ for $1\le i<k$,
\item $G[V_k]$ is in $\outcl n$ provided $k\ge 1$,
\end{enumerate}
 Then $G$ is  $\outcl {n+1}$.
\end{theorem}

\begin{proof}
By induction on $k$. For $k=0$ it is trivial. Assume now that the
statement is true for $k-1$ and prove it for $k$.

By (iii) $V_k=\outc {G[V_k]}n{A_k}$ for some independent sets
$A_k\subs V_k$. Put $V'=V\setm \outc G1{A_k}$ and let $V'_i=V_i\cap
V $ for $0\le i'< k$. Then we can apply the inductive hypothesis for
$G'=G[V']$ because (i) and (ii) imply that  the partition
$(V'_0,V'_1, \dots, V'_{k-1})$ satisfies (i)--(ii). Thus $V'$
contains an independent set $A$ such that $V'=\outc
{G[V']}{n+1}{A}$.

Let $\bar A=A\cup(A_k\setm \outc G1A)$. Then $\bar A$ is independent
because $\outc G1{A_k}\cap A\subs \outc G1{A_k}\cap V'=\empt$,
moreover $A_k\subs \outc G1{\bar A}$ and so $\outc G1{A_k}\subs
\outc G2{\bar A}$. Since $n+1\ge2$ it follows that $V=\outc
{G}{n+1}{\bar A}$.
\end{proof}
\noindent This result gives us the following direct generalization
of the Chv\'atal-Lov\'asz theorem:
\begin{corollary}\label{cor:mono}
If the chromatic number of  $G$ is finite then
$G\in \outcl 2$.
\end{corollary}
\begin{proof}
Indeed, the monochromatic classes are empty, so they are hereditary
in $\outcl 1$. Thus we can apply  Theorem \ref{tm:general2} to yield
$G\in \outcl 2$.
\end{proof}

\noindent Unfortunately Theorem \ref{tm:general2} does not give a
new proof to Theorem \ref{tm:cl} because for  finite graphs our
construction coincides with the original Chv\'atal-Lov\'asz argument.

The following theorem is  mainly a technical tool for later use
(proving Corollary \ref{cor:part}, Theorems \ref{tm:kn} and \ref{tm:clocally}).
\begin{theorem}\label{tm:general}
Let  $\ell,m\ge 1$ and assume that
 $G=(V,E)$ is  a directed graph having a partition
$(V_0,V_1,...,V_k)$ of $V$ such that
\begin{enumerate}[{\rm (i)}]
\item $G[V_0]$ is hereditary in  $\inout {m+1}{\ell+1}$,
\item $G[V_i]$ is hereditary in $\inout m\ell $ for $1\le i<k$,
\item $G[V_k]$ is in $\inout m\ell$ provided $k\ge 1$,
\end{enumerate}
 Then $G$ is in $\inout {m+1}{\ell+1}$.
\end{theorem}
\begin{proof}
We use 
induction on $k$. For $k=0$ the statement is
trivial. Assume that it is true for $k-1$ and prove it for $k$.

Let $(X_k,Y_k)$ be a $\inout m\ell$-partition of  $G[V_k]$, i.e.
$X_k=\outc {G[X_k]}\ell{A_k}$ and $Y_k=\inc {G[Y_k]}m{B_k}$ for some
independent sets $A_k$ and $B_k$. Put $V'=V\setm (\outc G1{A_k}\cup
\inc {G}1{B_k})$ and let $V'_i=V_i\cap V $ for $0\le i'< k$.

Then we can apply the inductive hypothesis for $G'=G[V']$ because
the partition $(V'_0,V'_1, \dots, V'_{k-1})$ satisfies (i)--(iii).
Thus $V'$ has a partition $(X,Y)$ and there are 
independent
sets $A\subs X$ and $B\subs Y$ such that $X=\outc
{G[X]}{\ell+1}{A}$ and $Y=\inc {G[Y]}{m+1}{B}$.

Let
\begin{eqnarray*}
\bar X & = &X_k\cup X\cup (\outc G1{A_k}\setm V_k),\\
\bar Y & = &Y_k\cup Y \cup (\inc G1{B_k}\setm (V_k\cup \outc
G1{A_k})),\\
\bar A & = & A\cup(A_k\setm \outc G1A),\\
 \bar B & = & B\cup(B_k\setm \inc G1B).
\end{eqnarray*}
Then $(\bar X,\bar Y)$ is a partition of $V$, $\bar A$ and $\bar B$
are independent, moreover
\begin{eqnarray*}
\bar X & = &\outc {G[\bar X]}{\ell+1}{\bar A} \qquad \hbox{and}\\
\bar Y & = &\inc {G[\bar Y]}{m+1}{\bar B}.
\end{eqnarray*}
\end{proof}
\noindent  When we started to study the problem (which later became
to Conjecture \ref{conjmain}) there were attempts to construct digraphs
$\not \in \inout 22 $ from ingredients like $\gra {\mathbb Z}<$. The
next statement shows that it is not possible:
\begin{corollary}\label{cor:part} If $G$ has a partition
$(A_1,\dots, A_k)$ such that each $G[A_i]$ is hereditary in $ \inout
11$ (for example $G[A_i]$ is isomorphic to  $\gra {\mathbb Z}<$, to
$\gra {\mathbb N}<$, to $\gra {\mathbb N}>$, or to a graph without
edges) then $G\in \inout 22$.
\end{corollary}
\begin{proof}
Since every $G[A_i]$ is hereditary in  $\inout 11$  we can apply
Theorem \ref{tm:general}.
\end{proof}

\section{When $G$ resembles to a tournament 
}\label{s:small}

\noindent
Let's recall that $\gra{\mathbb Z}<\not\in \outcl 2$ but clearly
$\gra{\mathbb Z}<\in \inout 11$. We  show  that a similar
theorem applies for an arbitrary tournament.

\begin{theorem}\label{tm:inft}If an infinite tournament  $G=(V,E)$
is not in $\outcl 2$ then $G\in \inout 11$.
\end{theorem}

\begin{proof}
Let $x\in V$ be arbitrary. If  $y\notin\outc {}2{x}$ then  $V=\inc
{}1x \cup \outc {}1y$.  Indeed,  if $z\notin \outc G1y$ then $(z,y)\in E$ but
$xzy$ is not a directed path of length two in $G$ by the choice of $y$, so
$(x,z)\notin E$. Thus $(z,x)\in E$, i.e $z\in \inc G1x $. Since $z$ was 
arbitrary, we obtain  $G\in \inout 11$.
\end{proof}

If $G=(V,E)$ is a digraph define the {\em undirected} complement of
the graph, $ \compl G=(V,\compl E)$ as follows: $\{x,y\}\in \compl
E$ if and only if $(x,y)\notin E$ and  $(y,x)\notin E$.
The graph $\compl G$ can be used to measures the difference between 
$G$ and a tournament: the more edges in $\compl G$ the large
difference between $G$ and a tournament; e.g.
$G$ is a tournament iff $\compl G$ does not have edges.

\begin{theorem}\label{tm:kn}
Let $G=(V,E)$ be a directed graph. If $K_n\not\subset \compl G$ for
some $n\ge 2$ then $G\in \inout 22$.  Moreover, if $n=2$ then $G\in
\outcl 2\cup \inout 11$, and if $n=3$, i.e. $\compl G$ is
triangle-free, then $G\in \inout 12$ or $G\in \inout 21$.
\end{theorem}

\begin{proof}
By induction on $n$. If $n=2$ then $\compl G$ does not contain
edges, i.e. $G$ is a tournament and so we are done by Theorem
\ref{tm:inft}.

Assume now that the theorem  is true for $n-1$ and prove it for $n$.
Let $A$ be a maximal independent set in $G$. If $V=\outc{}2{A}$ then
we are done.

If this is not the case, then let $C$ be a maximal independent set
in $G[V\setm \outc{}2{A}]$. Let $L=\inc{}1{A}\setm C$,
$M=\outc{}1{C}\setm L$ and $N=V\setm (L\cup M)$.

\psfrag*{c}{$c$}
\psfrag*{C}{$C$}
\psfrag*{a}{$a$}
\psfrag*{A}{$A$}

\psfrag*{V}{}
\psfrag*{L}{$L$}
\psfrag*{M}{$M$}
\psfrag*{N}{$N$}
\psfrag*{x}{$x$}
\psfrag*{y}{$y$}
\begin{center}
\includegraphics[keepaspectratio,width=5cm]{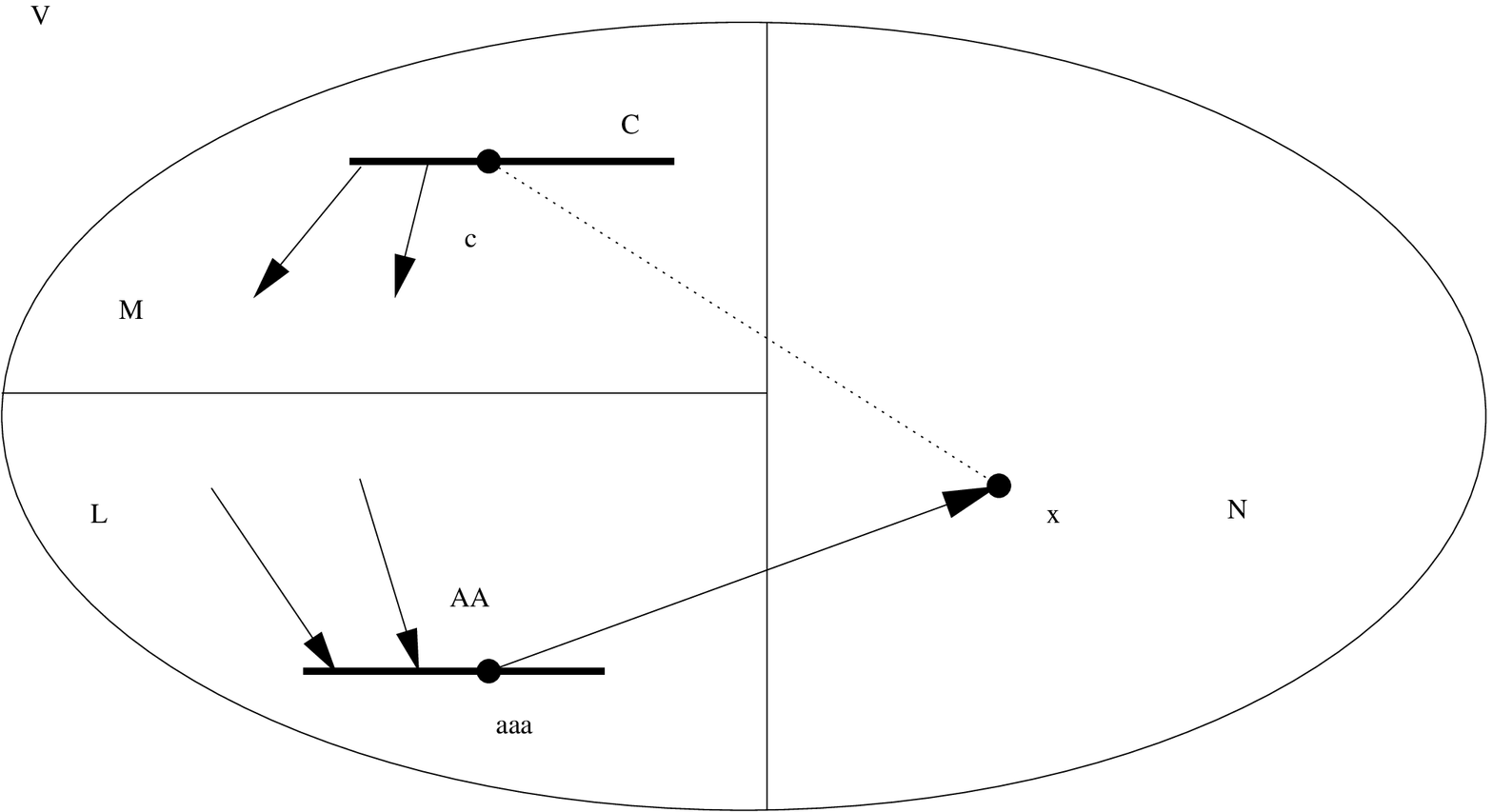}
\end{center}

\begin{claim}
There is no edge in $G$ between $N$ and $C$.
\end{claim}

\begin{proof}[Proof of the claim]
Let $x\in N$. If $a\in A$ then $(x,a)\notin E$ because
$x\notin \inc{}1{A}$ but $(a,x)\in E$ for some $a\in A$
because $A$ was maximal.
Moreover for each $c\in C$ we have $(c,x)\notin E$ because
$x\notin \outc{}1{C}$. But $(x,c)\notin E$ as well otherwise
the path $(a,x,c)$ witnesses that $c\in \outc{}2{A}$.
\end{proof}
\noindent Since $C\ne \empt$ we have that $K_{n-1}\not\subset
\compl{G[N]}$ (otherwise $\compl{G}$ would contain $K_n$). Hence we
can apply the inductive hypothesis for $G[N]$.
\begin{case}
Let $n=3$.
\end{case}
Then $G[N]$ is a tournament. If $N=\outc{G[N]}2d$ for some $d\in N$
then $L=\inc{G[L]}1{A}$ and $V\setm L=\outc{G[V\setm
L]}2{C\cup\{d\}}$. Thus $G\in \inout 12$.

Otherwise $N$ has a partition $P\cup R$ such that $P=\outc{G[P]}1x$
and $R= \inc{G[R]}1y$. Then
$$M\cup P=\outc{G[M\cup P]}1{C\cup \{x\})}$$  and
$$L\cup R=\inc{G[L\cup R]}2{\{y\}\cup\{a\in A:(a,y)\notin E\}}.$$
Thus $G\in \inout21$.

\begin{case}
Let $n>3$.
\end{case}
By the inductive hypothesis $G[N]$ is hereditary in $\inout 22$
(since $K_n \not\subset \compl{G}$ is a hereditary property),
moreover $G[L\cup M]\in \inout 11$, hence we can apply Theorem
\ref{tm:general} for $m=\ell=1$, for the digraph $G$ and for the
partition $(N,L\cup K)$ to yield $G\in \inout 22$.
\end{proof}

\begin{corollary}
Let $G=(V,E)$ be a directed graph.
If the chromatic number of  $\compl G$ is finite
then $G$  is $\inout 22$.
\end{corollary}
\noindent Indeed, if the chromatic number of $\compl G$ is $n$ then
$\compl G$ can not contain $K_{n+1}$.

\rem One can try to prove directly this corollary from Theorem
\ref{tm:general}. If the chromatic number of $\compl G$ is finite
then the vertex set has a partition $(V_),\dots, V_k)$ such that
every $G[V_i]$ is a tournament and so $G[V_i]$ is hereditary in
$\inout 12$. Thus applying directly Theorem \ref{tm:general} we can
get only $G\in\inout 23$.

\begin{theorem}\label{tm:clocally}
If $G=(V,E)$ is a digraph such that $\compl G$ is locally finite then
$G\in \inout 22$.
\end{theorem}

\begin{proof}
We prove by transfinite induction on ${\lambda}=|V|$. If ${\lambda}$
is finite then $G\in \outcl 2$ by Theorem \ref{tm:cl}. We can assume
that ${\lambda}=|V|$ is infinite and we have proved the theorem for
graphs of cardinality $<{\lambda}$. We distinguish two cases.

\smallskip\noindent{\bf Case 1:}{\em  There is $\{x,y\}\in \br V;2;$
such that the set  $U=\outc G1x\cap \inc G1y$ has cardinality
${\lambda}$.}

We will find a partition $(X,Y)$ of $V$ such that $X=\outc{G[X]}2x$
and $Y=\inc{G[Y]}2y$. For that end fix an enumeration
$\<v_{\zeta}:{\zeta}<{\lambda}\>$ of $V$. By transfinite induction
on ${\zeta}<{\lambda}$ we will construct disjoint subsets
$X_{\zeta}$ and $Y_{\zeta}$ of $V$, $|X_{\zeta}|+|Y_{\zeta}|
<{\omega}+|{\zeta}|$, such that $X_{\zeta}=\outc{G[X_{\zeta}]}2x$ and
$Y_{\zeta} =\inc{G[Y_{\zeta}]}2y$.

Put $X_0=\{x\}$ and $Y_0=\{y\}$. Assume that for all $\eta < \zeta$
we have already constructed $X_\eta, Y_\eta.$ If ${\zeta}$ is a
limit ordinal put $X_{\zeta}=\bigcup\{X_{\xi}: {\xi}<{\zeta}\}$ and
$Y_{\zeta}=\bigcup\{Y_{\xi}:{\xi}<{\zeta}\}$

If $\zeta$ is not a limit ordinal, then $\zeta=\eta +1$ and we have
$X_{\eta}$ and $Y_{\eta}$ such a way that $X_\eta=\outc
{G[X_\eta]}2x$ and $Y_\eta=\inc{G[Y_\eta]}2y$. Let
$i=\min\{i':v_{i'}\notin X_{\eta}\cup Y_{\eta}\}$.

If $|\inc G1{v_i}\cap U|={\lambda}$ then let
$$j=\min\{j':v_{j'}\in (\inc G1{v_i}\cap\outc G1x\setm
(X_{\eta}\cup Y_{\eta})\},
$$
and let $X_{\zeta}=X_{\eta}\cup\{v_i,v_j\}$ and
$Y_{\zeta}=Y_{\eta}$.

If $|\inc G1{v_i}\cap U|<{\lambda}$ then $|\outc G1{v_i}\cap
U|={\lambda}$ because
$v_i$ has finite degree in $\compl G$.
Let
$$j=\min\{j':v_{j'}\in \left(\outc G1{v_i}\cap\inc G1y\right)\setm (X_{\eta}
\cup Y_{\eta})\},$$ and let $Y_{\zeta}=Y_{\eta} \cup\{v_i,v_j\}$ and
$X_{\zeta}=X_{\eta}$. Put finally $X=X_{\lambda}$ and
$Y=Y_{\lambda}$.

\smallskip\noindent{\bf Case 2:}  {\em $|\outc G1x\cap \inc
G1y|<{\lambda}$ for each $\{x,y\}\in \br V;2;$}.

Let $\{x,y\}\in \br V;2;$ be arbitrary vertices. Put $W=V\setm
(\outc G1x\cup \inc G1y)$. Then $W\setm (\inc G1x\cap \outc G1y)= (W
\setminus \inc G1x) \cup (W\setminus \outc G1y)$ is finite because
$\compl G$ is locally finite. Thus $|W|<{\lambda},$ hence $G[W]$ is
hereditary in $\inout 22$ by the inductive hypothesis. Moreover
$V\setm  W=\outc G1x\cup \inc G1y$, hence $G[V\setm W]\in \inout
11$. Therefore  we can apply Theorem \ref{tm:general} for
$m=\ell=1$, the digraph $G$ and the partition $(W,V\setm W)$ to
yield $G\in \inout 22$.
\end{proof}

\section{Infinite tournaments}\label{s:tour}

\noindent
In this section we prove structure theorems for infinite tournaments.

  For any cardinal ${\kappa}$ let 
the digraph $\testinfk=\gra {\kappa}\ge$, i.e.
$(x,y)$ is an edge if and only if $x\ge y$. (This is a tournament
with all possible loops.)


\begin{theorem}\label{tm:non-inft}
For an infinite tournament $G=(V,E)$ the followings are equivalent:
\begin{enumerate}[{\rm (i)}]
\item $G\notin\outcl \infty$,
\item  for some  regular cardinal ${\kappa}$ 
there is a surjective homomorphism ${\varphi}:G\to \testinfk$.
 \end{enumerate}
\end{theorem}

\begin{proof}
(ii) clearly implies (i): if ${\varphi}(x)=k$ then ${\varphi}(y)\le
k$ for each $y\in \outc G{\infty}x$, and so $\outc G{\infty}x\ne V$
because ${\varphi}$ is surjective.

Assume now that (i) holds, i.e. $G\notin \outcl \infty$. Then define
the equivalence relation $\equiv$ on $V$  as follows:
\begin{equation}
x\equiv y\quad \Longleftrightarrow \quad \Big ( x\in \outc G\infty y
\ \hbox{ and }\  y\in \outc G\infty x \Big ).
\end{equation}
Let us denote by $\ecal$ the equivalence classes $[x]_\equiv$ of the
equivalence relation $\equiv$. Now define the relation $\preceq$ on
$\ecal$ as follows:
\begin{equation}
X\preceq Y \Longleftrightarrow \Big ( X=Y\ \hbox{ or }\ (y,x)\in E \
\hbox{ for each pair}\ x\in X,   y\in Y \Big ).
\end{equation}
Clearly $X\preceq Y$ if $(y,x)\in E$ for some $x\in X$ and $y\in Y$.
The relation $\preceq$ is an ordering on $\ecal$. If $\preceq$ has a
last element $X$ then $V=\outc G\infty x$ for each $x\in X$. Hence
$\<\ecal,\preceq\>$ contains a strictly increasing cofinal sequence
$\<X_{\xi}:{\xi}<{\kappa}\>$ for some regular cardinal ${\kappa}$.

Define ${\varphi}:V\to {\mathbb{N}}$ by the formula
${\varphi}(v)=\min \{{\xi}: [v]_\equiv\preceq X_{\xi}\}$. The map
${\varphi}$ is clearly  homomorphism onto $\testinfk$.
\end{proof}

\noindent Define the digraph  $\testt=\<\mathbb N,E\>$ as follows
\begin{equation}
E=\{(x,y):x\ge y\}\cup\{
(x,x+1):x\in \mathbb N\}.
\end{equation}
$\testt$ can be obtained from $\testinf$ by adding edges
$\{(n,n+1):n\in \mathbb N\}$.

\bigskip
\psfrag*{x0}{0}
\psfrag*{x1}{1}
\psfrag*{x2}{2}
\psfrag*{x3}{3}
\psfrag*{testinf}{$\testinf$}
\psfrag*{testt}{$\testt$}
\begin{center}
\includegraphics[keepaspectratio,width=10cm]{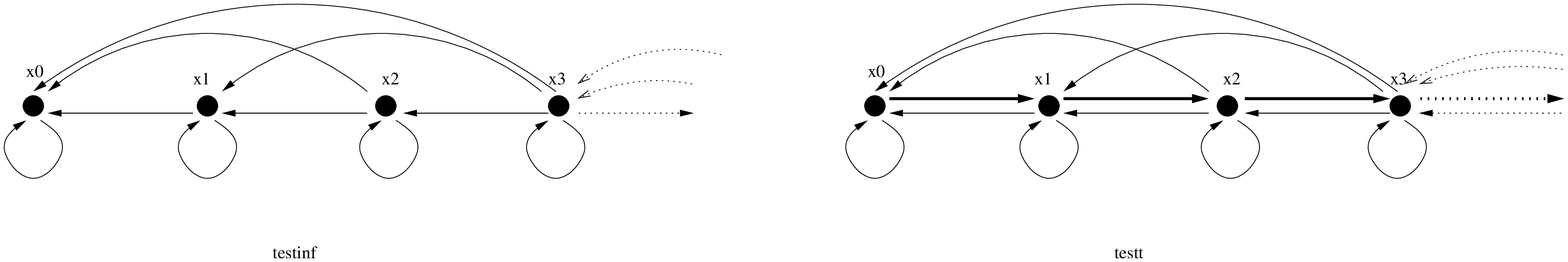}
\end{center}

\begin{theorem}\label{tm:nott} For an
 infinite tournament $G\in \outcl \infty$ the followings are equivalent:
 \begin{enumerate}[{\rm (i)}]
\item
$G\notin\outcl 3$,
\item $G\notin\outcl n$ for each $n\ge 3$,
\item there is a surjective homomorphism ${\varphi}:G\to \testt$.
 \end{enumerate}
\end{theorem}

\begin{proof}
(iii) clearly implies (ii): if ${\varphi}(x)=k$ then
${\varphi}(y)\le k+n$ for each $y\in \outc Gnx$.

To prove that (i) implies (ii) assume  that $G\in \outcl n$ for some
$n\ge 3$ and we show $G\in \outcl 3$. Fix $x\in V$ and $n\ge 3$ such
that $V=\outc Gnx$ but $V\ne \outc G{n-1}x$. Pick $y\in \outc
G{n}x\setm \outc G{n-1}x$. We claim that $V=\outc G3y$. Indeed,
$\outc G{n-2}x\subs \outc G1y$ because $y\notin \outc G{n-1}x$. 
Hence  $\outc G{n-1}x\subs \outc
G1{\outc G{n-2}x}\subs \outc G2y$ and so finally we obtain that
$V=\outc G{n}x\subs \outc G1{\outc G{n-1}x}\subs \outc G3y$.

Assume finally that (ii) holds. 
Since $G\in \outcl \infty$   there is
$x\in V$ with $V=\outc G\infty x$. Then define ${\varphi}:V\to
{\mathbb{N}}$ as follows: ${\varphi}(y)=\min\{n:y\in \outc Gnx\}$.
${\varphi}$ is clearly a homomorphism and it is onto because $\outc
G1n\ne V$ for $n\in \mathbb N$.
\end{proof}

\begin{theorem}\label{tm:o3no2}
There is an  infinite tournament $T\in \outcl 3 \setm \outcl 2$.
\end{theorem}
\noindent The proof is based on a construction method we will
develop in Section \ref{s:rec} so it will be presented  just after
Theorem \ref{tm:go3no2B}.

\begin{problem}
Find a characterization of $G\notin \outcl 2$  a la {\rm Theorem
\ref{tm:nott}}.
\end{problem}

\section{Infinite digraphs  generated by a finite structure}
\label{s:rec}
\noindent
To prove Theorem \ref{tm:o3no2} we develop a recursive method to
construct infinite digraphs from certain finite ones and we
investigate the properties of the graphs which can be obtained in
this way.

\begin{definition}
A {\em \grt} is a triple $G=(V,E,T)$, where $(V,E)$ is a digraph
and $\empt\ne T\subs V$. The elements of $T$ are the
{\em terminal vertices of $G$}, the elements of $V\setm T$ are the
{\em nonterminal vertices of $G$}.
For a \grt\ $G=(V,E,T)$ write
$V_G=V$, $E_G=E$, $T_G=T$ and $N_G=V_G\setm T_G$.
\end{definition}

Assume that we have a \grt\ $G=(V,E,T)$. Write $N=V\setm T$.
Construct a new \grt\ $G\bigodot G=(W,F,S)$ from $G$ as follows:
keep the terminal vertices and blow up  each nonterminal vertex $v$
to a (disjoint) copy of $G$ denoted by $G_v$.  So we can set
\begin{equation}
W=T\cup (N\times V)
\end{equation}
The edges are ``inherited'' from $G$
in the natural way:
\begin{equation}
F=\{(x,y):(x(\Delta(x,y)), y(\Delta(x,y)))\in E\},
\end{equation}
where $\Delta(x,y)=\min\{i:x(i)\ne y(i)\}$. (Here every vertex $x$
can be described with a length at most 2 sequence: $x(0)$ gives the
position of our point $x$ within the original $G$ copy, while $x(1)$
gives its position within the new $G$ copy inserted into $x(0)$ if
this was not a terminal point.)

Define terminal and non-terminal vertices of the digraph $G \bigodot
G$ in the natural way: $S=T\cup (N\times T)$ is the set of the
terminal vertices. Hence  $N\times N$ is the set of the non-terminal
vertices of  $G\bigodot G$.
\begin{center}
\psfrag{g}{$G$}
\psfrag{gdotg}{$G\bigodot G$}
\includegraphics[keepaspectratio, width=10cm]{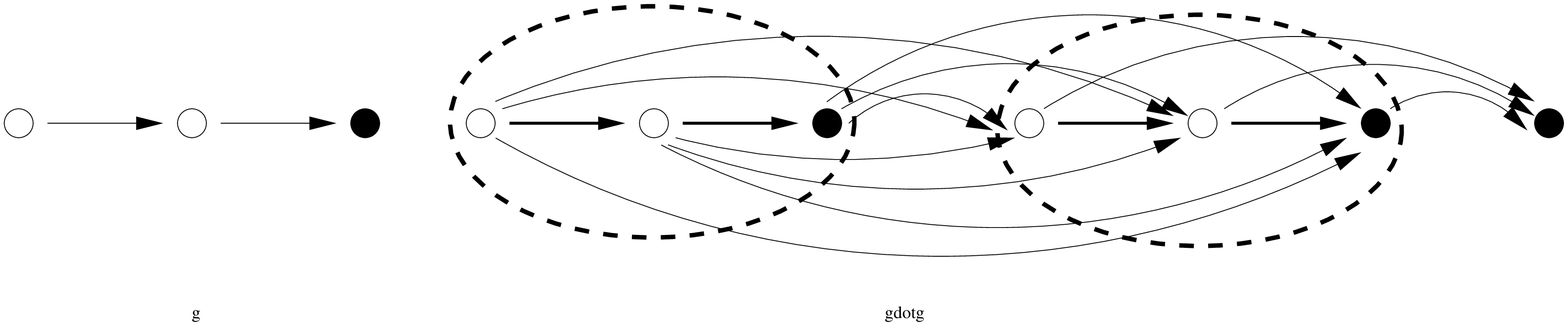}
\end{center}

\noindent Observe that
\begin{equation}\label{javit}
G[T]\text{ is a induced subgraph of }(G\bigodot G)[S].
\end{equation}
Now we can repeat the procedure above using  $G\bigodot G$ instead
of $G$ to get $(G\bigodot G)\bigodot G$. (Here the history of every
vertex can described with a length at most 3 sequence.)  Hence we
obtain a sequence $\<G_n:n\in \mathbb N\>$ of \grts,
$G_n=\<V_n,E_n,T_n\>$ such that
\begin{equation}
G_0[T_0]\subs G_1[T_1]\subs G_2[T_2]\subs \dots
\end{equation}
Take
\begin{equation}
G^\infty=\bigcup\{G_n[T_n]:n\in\mathbb N\}.
\end{equation}
This was the informal definition of $G^\infty$. The formal
definition  is much shorter:
\begin{definition}
If $G=(V,E,T)$ is  \grts, $N=V\setm T$ then
define the digraph $G^\infty=(N^*\concat T, F)$ as follows:
\begin{equation}\notag
F=\{(x,y):(x(\Delta(x,y)), y(\Delta(x,y)))\in E\},
\end{equation}
where $\Delta(x,y)=\min\{i:x(i)\ne y(i)\}$. \hfil \qed
\end{definition}

\noindent We will write sometimes $V^\infty$ instead of
$V(G^\infty)$ and $E^\infty$ instead of $E(G^\infty).$ First we
prove two theorems which will give the example needed in Theorem
\ref{tm:o3no2}.
\begin{theorem}\label{tm:go3no2A}
Let $G=(V,E,T)$ be a finite \grt. Then
the followings are equivalent:
\begin{enumerate}[{\rm (i)}]
\item  $G^\infty\in \outcl 3$,
\item  $G^\infty\in \outcl \infty$,
\item  $\inc {}1v\ne \{v\}$ for each
$v\in V\setm T$.
\end{enumerate}
\end{theorem}
\begin{proof} It is clear that (i) implies (ii).

Assume that (iii) fails: i.e. $\inc G1v=\{v\}$ for some $v\in V\setm
T$. Define ${\varphi}:V(G^\infty) \to {\mathbb{N}}$ as follows:
${\varphi}(s)=\min \{n:s(n)\ne v\}$. Then ${\varphi}$ is a
surjective homomorphism from $G^\infty$ onto $\testinf$, so
$G^\infty \notin \outcl \infty$ (See Theorem~\ref{tm:non-inft}).
Thus (ii) implies (iii).

Assume (iii). For $v\in V$ define $v'\in V^\infty$ as follows:
$v'=v$ for $v\in T$ and $v'=v\concat t$ for $v\in V\setm T$, where
$t\in T$ is arbitrary.

Let $A\subs V$ be independent such that $V=\outc G2A$. Put
$K=\{a':a\in A\}$. Then $K$ is clearly independent in $G^\infty$. We
claim that $V^\infty=\outc {G^\infty}3{K}$.

Let $x\in V^\infty$. Write $x=y\concat s$, where $s\in T$.   If
$y=\empt$ then either $s\in A$ and so $s'=s\in K$ as well, or
$s\notin A$ and so $s\in \outc G2a$ for some $a\in A$. If $(a,s)\in
E$ then $(a',s)\in E^\infty$. If $abs$ is a directed path of length
$2$ in $G$ then $a'b's$ is a directed path of length $2$ in
$G^\infty$ In any case  $x=s\in \outc {G^\infty}2{a'}$.

Assume now that $y\ne\empt$. Then, by (iii), there is $w\in V$ such
that $(w,y(0))\in E$. Then $(w',x)\in E^\infty$. If $w\in A$ then
$w'\in K$ and $(w',x)\in E^\infty$. If $w\notin A$ then $w\in \outc
G2a$ for some $a\in A$. If $(a,w)\in E$ then $(a',w')\in E^\infty$.
If $abw$ is a directed path of length $2$ in $G$ then $a'b'w'$ is a
directed path of length $2$ in $G^\infty$ In any case  $w'\in \outc
{G^\infty}2{a'}$ and so $x\in \outc {G^\infty}3{a'}$.
\end{proof}

\begin{theorem}\label{tm:go3no2B}
Let $G=(V,E,T)$ be a finite \grt. 
If $(V,E)$ is a 
tournament then
the followings are also equivalent:
\begin{enumerate}[{\rm (i)}]
\item $G^\infty \in \outcl 2$,
\item  there is $v\in T$ with $V=\outc G2v$.
\end{enumerate}
\end{theorem}
\begin{proof}
First of all observe that (ii) clearly implies (i): if  $V=\outc
G2v$ for some $v\in T$ then $V^\infty=\outc {G^\infty}2v$.

Assume now that (ii) fails and let $s\in V^\infty$ be arbitrary. We
will show that $\outc {G^\infty}2s\ne V^\infty.$ Let $s=r\concat t$
where $r\in (V\setm T)^n$ and $t\in T$. Let $X=\{r' \in V^\infty: r'
\restriction n =r\}$.

Then  $\outc G2t\ne V$ implies $\outc {G^\infty[X]}2s\ne X$. Pick an
arbitrary $y\in X\setm \outc {G^\infty[X]}2s.$ Then $y$ cannot be
reached with a length at most two directed path from $s$ within
$G^\infty[X].$ Similarly, if $p\notin X$ then $(s,p)\in E^\infty$ if
and only if $(y,p)\in E^\infty$ (since $\Delta(s,p)=\Delta(y,p)$),
therefore the vertex triplet $spy$ can not form a directed path of
length 2 from $s$ to $y$ in $G^\infty$. Hence $y\notin \outc
{G^\infty}2s$.
\end{proof}

\begin{proof}[{\bf Proof of Theorem \ref{tm:o3no2}}]
Consider the following  \grt: $G=(\{0,1,2,3\},E,\{0\})$, where
$$E=\{(0,1), (1,2), (2,3), (3,1), (3,0), (2,0)\}.$$
\nopagebreak
\psfrag*{y0}{0}
\psfrag*{y1}{1}
\psfrag*{y2}{2}
\psfrag*{y3}{3}
\begin{center}
\includegraphics[keepaspectratio,width=5cm]{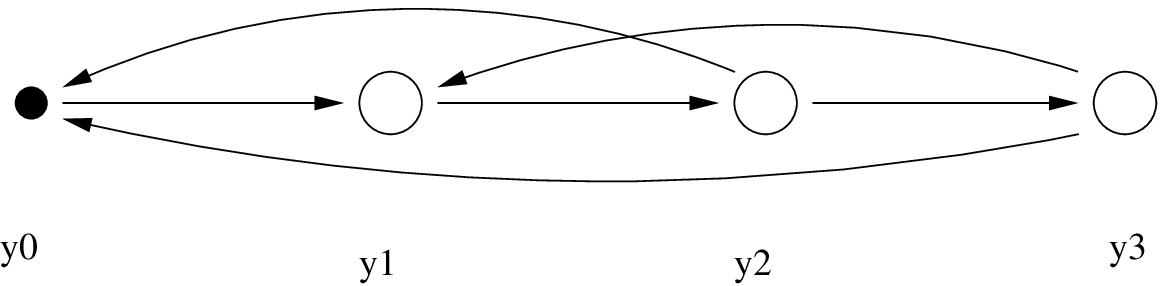}
\end{center}
$G$ is a finite tournament, so 
by Theorems \ref{tm:go3no2A} and \ref{tm:go3no2B} the tournament
$G^\infty\in \outcl 3\setm \outcl 2$.
\end{proof}

\begin{theorem}\label{tm:ginf}
If $G=\<V,E,T\>$ is a finite \grt\ then $G^\infty\in \inout 22$.
\end{theorem}
\noindent We prove this theorem through a series of theorems and
lemmas.

\begin{theorem}\label{tm:nfree}
If $G=\<V,E,T\>$ is a finite \grt\ such that $V\setm T$ is
independent in $G$ then  $G^\infty\in \incl 2\cup \outcl 2$.
\end{theorem}
\begin{proof}[Proof of Theorem \ref{tm:nfree}]
Write $N=V\setm T$. We introduce two properties of $G$.
\begin{itemize}
\item[(I)]  $T\subs \outc G1N\cap \inc G1N$,
\item[(II)] $N\not\subs\outc G1T$ and $N\not\subs \inc G1T$.
\end{itemize}
\noindent We start with auxiliary lemmas.
\begin{lemma}\label{lm:at}
If there is an independent set $A\subs V$ such that $\outc G2A=V$
$[\inc G2A=V]$ and $T\cap A\ne \empt$ then $G^\infty\in \outcl 2$
$[G^\infty\in \incl 2]$.
\end{lemma}
\begin{proof}[Proof of Lemma \ref{lm:at}]  \prlabel{lm:at}
We will show $K=(A\cap N)^*\concat (A\cap T)$ is a quasi-kernel in
$G^\infty$. The set $K$ is clearly independent in $G^\infty$ because 
$A$ was independent in $G$.
 
Fix an element  $t\in T\cap A$.
For $x\in T$ let $\ext x=x$ and for $x\in  N$ let $\ext x=x\concat
t$.

Let $s\in V^\infty$. We show that $s\in \outc {G^\infty}2K$.
If $s\in A^*$ then $s\in K$, so we can assume that 
$s=s'\concat p$, where $p(0)\notin A$. 
Then there is $a\in
A$ such that either $(a,p(0))\in E $ or there is  $x\in V$ such that
$(a,x)\in E$ and $(x,p(0))\in E$. Then $s'\concat \ext {a}\in K$ and
in the first case $(s'\concat \ext {a}, s)$ is an edge in 
$ G^\infty$, and in the
second case $(s'\concat \ext {a}, s'\concat \ext {d},s)$ is a
directed path of length 2 in $G^\infty $. Therefore $s\in \outc
{G^\infty}2{s'\concat \ext {a}}\subs \outc {G^\infty}2{K}$.
\end{proof}
\begin{lemma}\label{lm:tio}
If {\rm (I)} fails then $G^\infty\in \incl 2\cup \outcl 2$, more
precisely, if   $T\not\subs \outc G1N$ $[T\not\subs   \inc G1N]$
then $G^\infty\in \outcl 2$ $[G^\infty\in \incl 2]$.
\end{lemma}

\begin{proof}[Proof of Lemma \ref{lm:tio}]\prlabel{lm:tio}
Assume that $t\in T\setm \outc G1N$. Let $B=\outc G1t$. Let $A'\subs
V\setm B$ be independent such that $V\setm B=\outc {G[V\setm
B]}2{A'}$. If there is no edge between $A'$ and $\{t\}$ then $A=A'\cup
\{t\}$ satisfies the assumptions of Lemma \ref{lm:at}. Hence
$G^\infty  \in \outcl 2$.

If there is an edge edge between $A'$ and $\{t\}$, 
then there is $a\in A'$ with $(a,t)\in E$
because $\outc G1t\cap B=\empt$. Thus $t\in \outc G1a$ and so
$a\notin N$, i.e. $a\in T$.
Hence $A'$ satisfies the assumptions of Lemma
\ref{lm:at}. Hence $G^\infty  \in \outcl 2$.
\end{proof}
\begin{lemma}\label{lm:nt}
If $G$ satisfies  {\rm (I)} but {\rm(II)} fails then $G^\infty\in
\incl 2\cup \outcl 2$, more precisely, if $N\subs \inc G1T$ $[N\subs
\outc G1T]$ then $G^\infty \in \incl 2$ $[G^\infty \in \outcl 2]$.
\end{lemma}
\begin{proof}[Proof of Lemma \ref{lm:nt}]
\prlabel{lm:nt} Let $t\in T$ be a fixed element and put
$K=\{y\concat t:y\in N\}$. $K$ is clearly independent. Let $s\in
V^\infty$. If $s\in T$ then by assumptions (I) we have $(s,y)\in E$ for some
$y\in N$ and so $s\in \inc {G^\infty}1{y\concat t}\subs \inc {G^\infty}1K$.

If $s(0)=x\in N$ then there is $s\in T$ with $(x,t)\in E$ by the
assumption $N\subs \inc G1T$. Then, by (I), there is $y\in N$ with
$(t,y)\in E$. Thus $xty$ is a directed path of length $2$
in $G$ and so $s\in \inc
{G^\infty}2{y\concat t}$ and $y\concat t \in K$.
\end{proof}
\noindent By Lemmas  \ref{lm:tio} and \ref{lm:nt} we can assume that
$G$ has properties (I) and (II). Let $A=N\setm \outc G1T$ and
$B=N\setm \inc G1T$. Hence $A\ne \empt$ and $B\ne \empt$ by (II).
Now (I) implies that $A\ne N$ and $B\ne N$ because
$N\cap \outc G1T\ne \empt$ yields $T\cap \inc G1N\ne \empt$.

Let $b\in B$ and $t\in T$ be fixed, and put
 $K=A^*\concat (N\setm A)\concat t$.
If $\{p,q\}\in \br K;2;$ then we have
$\{p(\Delta(p,q)),q(\Delta(p,q)\}\in \br N;2;$ and so there is no
edge between $p$ and $q$ in $G^\infty$. Hence $K$ is independent.

Let $L=A^*\concat T$. Now we have
\begin{equation}\label{eq:kl}
\outc {G^\infty}1K\supset L.
\end{equation}
Indeed, if $x\in A^*$ and  $s\in T$ then there is $c\in N$ with
$(c,s)\in E$. If $c\in N\setm A$ then $x\concat c \concat t\in K$
and $(x\concat c\concat t, x\concat s)\in E^\infty$. If $c\in A$
then  $x\concat c \concat b\concat t\in K$ and $(x\concat c\concat
b\concat t, x\concat s)\in E^\infty$

Moreover, we claim that 
\begin{equation}\label{eq:lv}
\outc {G^\infty}1L=V^\infty.
\end{equation}
Indeed, let $x\in V^\infty\setm L$. Let $n$ be maximal such that
$x\restriction n\in A^n$. Since $x\notin L$ we have $c=x(n)\in
N\setm A$. Hence $c\in \outc G1T$, so we can pick $s\in T$ with $(s,c)\in E.$
Then $(x\restriction n)\concat s\in L$ and $(x\restriction n\concat
s,x)\in E^\infty$. 

Hence $\outc {G^\infty}2K\supseteq \outc {G^\infty}1L=V^\infty$, i.e  
$G^\infty \in \outcl 2$. This concludes the proof of Theorem 
\ref{tm:nfree}.
\end{proof}

\begin{proof}[Proof of Theorem \ref{tm:ginf}]\prlabel{tm:ginf}
By Theorem \ref{tm:nfree} we can assume that there is an edge
$(x,y)\in E\cap (N\times N)$.

Let $A=
\inc G1{\{x,y\}}\cup \outc G1{\{x,y\}}$.
Let $B\subs V\setm A$ be independent such
that $\outc {G[V\setm A]}2B=V\setm A$. Write $B_N=B\cap N$ and
$B_T=B\cap T$. Let $t\in  T$ be fixed. Let
\begin{displaymath}
K_0=B_N{}^*\concat xt,\ K_1=B_N{}^*\concat B_T\quad \text{ and  }
\quad K=K_0\cup K_1,
\end{displaymath}
and
\begin{displaymath}
L=B_N{}^*\concat yt.
\end{displaymath}
Let
\begin{displaymath}
V_0=(B_N{}^*\concat A\concat V^*)\cap V^\infty.
\end{displaymath}
Let
\begin{displaymath}
 R=K_0\cup (B_N{}^*\concat yxt)
\end{displaymath}
and
\begin{displaymath}
 S=L\cup (B_N{}^*\concat xyt).
\end{displaymath}
Clearly $R\cup S\subs V_0$ and $R\cap S=\empt$, moreover 
\begin{equation}\label{kern}
\text{$R\subs
\outc {G^\infty}1{K_0}$ and $S\subs \inc {G^\infty}1{L}$  }
\end{equation}
because $(s\concat xt,s\concat yxt)\in E^\infty$ and
$(s\concat xyt,s\concat syt)\in E^\infty$ for each $s\in B_N{}^*$.

\begin{claim}\label{cl:v0}
$V_0= \outc {G[V_0]}1R\cup \inc {G[V_0]}1S$.
\end{claim}

\begin{proof}
\prlabel{cl:v0} Let $s\in V_0$ be arbitrary. Write $s=s'\concat
p$, where $s'\in B_N{}^*$ and $p(0)\in A$.
If $p(0)=x$ then $(s,s'\concat yt)\in
E^\infty$ and so $s\in \inc {G^\infty[V_0]}1{s'\concat yt}\subs \inc
{G^\infty[V_0]}1{S}$. Similarly, if $p(0)=y$ then $s\in \outc
{G^\infty[V_0]}1{R}$. So we can assume that  $a=p(0)\notin \{x,y\}$.
Then $(x,a)\in E$ or $(a,x)\in E$ or $(a,y)\in E$ or $(y,a)\in A$.
If $(x,a)\in E$ then $(s'\concat xt,s)\in E^\infty$ and so $s\in
\outc {G^\infty[V_0]}1R$. If $(a,x)\in E$ then $(s, s'\concat
xyt)\in E^\infty$ and so $s\in \inc {G^\infty[V_0]}1S$.
The remained cases can be handled similarly.
\end{proof}
\noindent

Let $R'=\outc {G[V_0]}1R\setm S$ and $S'=V_0\setm R'$.
Then $(R',S')$ is clearly a partition of $V_0$, 
and (\ref{kern}) and Claim \ref{cl:v0} together
imply that 
\begin{equation}
\text{$R'=\outc
{G^\infty[R']}2{K_0}$ and $S'=\inc {G^\infty[S']}2{L}$.}  
\end{equation}

  Now let
$V_1=(V^\infty\setm V_0)\cup K_0$.

\begin{claim}\label{cl:v1}
$V_1=\outc {G^\infty[V_1]}2{K}$.
\end{claim}

\begin{proof}
\prlabel{cl:v1} For $z\in T$ let $\ext z=z$ and for $z\in  N$ let
$\ext z=z\concat xt$. Let $s\in V^\infty\setm V_0$. Write
$s=s'\concat p$, where $s'\in B_N{}^*$ and $p(0)\notin B_N{}^*$.
Since $s\notin V_0$ we have $p(0)\notin A$. If $p(0)\in B_T$ then
$s=s'\concat p(0)\in K$.

Hence we can assume that $p(0)\in V\setm (A\cup B).$ Thus there is
$b\in B$ such that either $(b,p(0))\in E $ or there is  $z\in V\setm
A$ such that $(b,z)\in E$ and $(z,p(0))\in E$. Then $s'\concat \ext
{b}\in K$ and in the first case $(s'\concat \ext {b}, s)\in
E^\infty$ and in the second case $(s'\concat \ext {b}, s'\concat
z,s)$ is a directed path of length 2 in $G^\infty $.
Therefore $s\in \outc {G^\infty[V_1]}2{s'\concat \ext {b}}\subs
\outc {G^\infty[V_1]}2{K}$.
\end{proof}
\noindent Hence the partition $(V^\infty\setm S',S')$ witnesses that
the digraph $G^\infty$ is in $ \inout 22$: $V^\infty\setm S'=R'\cup V_1= \outc
{G^\infty[R'\cup V_1]}2K$ and $S'=\inc {G^\infty[S']}2{L}$.
\end{proof}

\section{An observation}\label{s:haj}

\noindent
Consider once more the tournament $G=\gra {\mathbb Z}<$. Although $G \notin
\outcl 2$ (even $G\notin \outcl \infty\cup \incl \infty$), $G$ has
two vertices, $a=1$ and $b=0$, such that $\mathbb Z=\outc G1a\cup
\inc G1b$. This situation is not unique among the infinite digraphs:

\begin{theorem}\label{tm:o2i2}
For each  directed graph $G=\<V,E\>$ there are disjoint, independent
subsets $A$ and $B$ of $ V$ such that $V=\outc {}2A \cup \inc {}2B$.
\end{theorem}
\begin{proof}
Let $F_0$ be maximal independent subset in $G$, and let $F_1$ be
maximal independent subset in $G[V\setminus \inc{}1{F_0}]$. Put
$A=F_0\cap \inc{}1{F_1}$ and $B=F_1\cup (F_0\setminus A)$.
\vspace{0.5cm} \psfrag*{f0}{$F_0$} \psfrag*{a}{$A$}
\psfrag*{vmb1f0}{$V\setminus \inc{}1{F_0}$} \psfrag*{b}{$B$}
\psfrag*{f1}{$F_1$} \psfrag*{b1f0}{$\inc{}1{F_0}$}
\psfrag*{b1f1}{$\inc{}1{F_1}$} \psfrag*{k1f1}{$\outc{}1{F_1}$}
\begin{center}
\includegraphics[keepaspectratio,width=10cm]{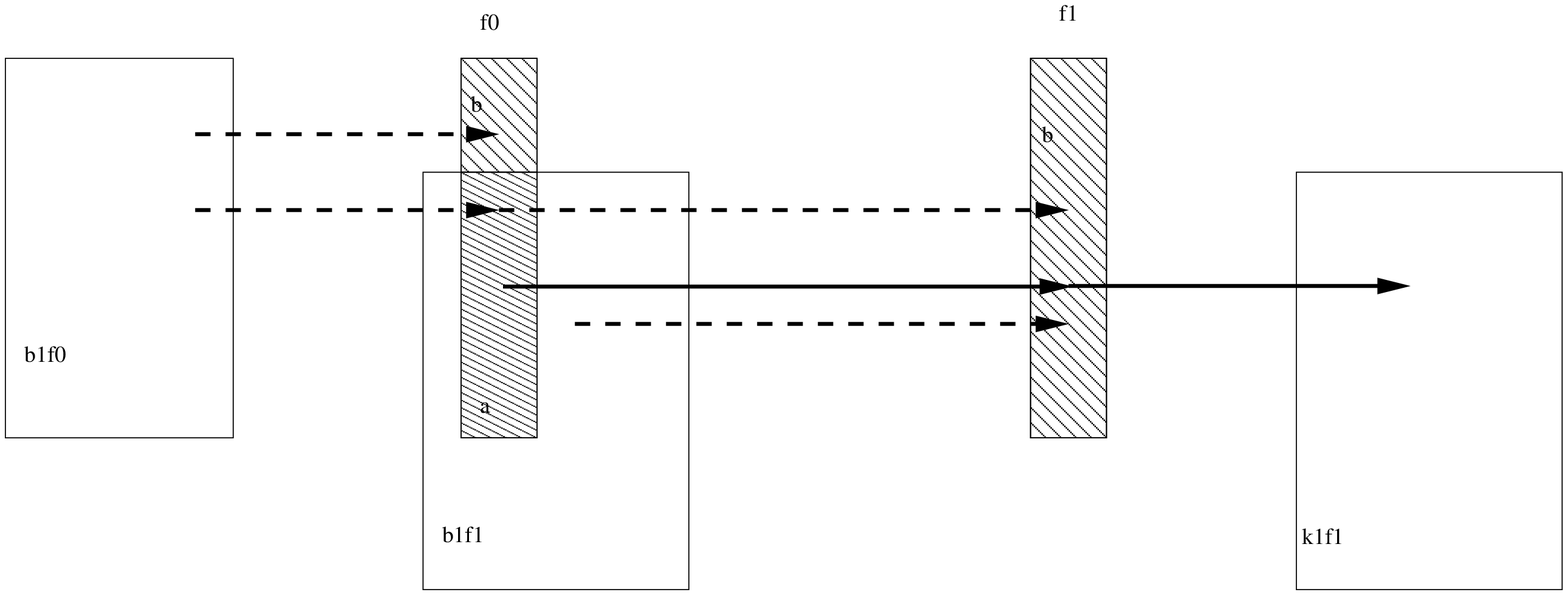}
\end{center}
The sets $A$ and $B$ are clearly independent. Moreover
\begin{eqnarray}\label{e:if0}
\inc {}1{F_0}=
\inc {}1{F_0\cap \inc {}1{F_1}}\cup \inc {}1{F_0\setm \inc {}1{F_0}}
\subs \nonumber\\
\inc {}2{F_1}\cup \inc {}1{B}\subs \inc {}2{B}.
\end{eqnarray}
Since $F_1\subs \outc {}1{A}$ and so $\outc {}1{F_1}\subs \outc
{}2{A}$ we have
\begin{eqnarray}\label{e:of0}
V\setm  \inc {}1{F_0}\subs \outc {}1{F_1}\cup \inc {}1{F_1} \subs
\outc {}2{A}\cup \inc {}1B\subs \nonumber\\ \outc {}2{A}\cup \inc
{}2{B}.
\end{eqnarray}
(\ref{e:if0}) and (\ref{e:of0}) together yield $V=\outc{}2{A}\cup
\inc{}2{B}$.
\end{proof}

Unfortunately, the construction above can not be applied to  solve 
Conjecture \ref{conjmain}
because if $V=\outc {G[V_A]}A2\cup \inc {G[V_B]}B2$ for some 
$V_A,V_B\subs V$ then 
we should have $A\subs  V_B$ otherwise we could not guarantee
$\inc {}1{F_0}\setm  \inc {}1{B}\subs \inc {G[V_B]}B2$.

\end{document}